\documentclass[11pt, reqno]{amsart}
\usepackage[all,tips]{xy}
\usepackage{latexsym, amsfonts, amsmath, amssymb, mathrsfs, graphicx, xcolor, ytableau, hyperref, float}
\usepackage[justification=centering]{caption}
\allowdisplaybreaks
\usepackage{centernot}
\usepackage{tikz}

\definecolor{red}{rgb}{1,0,0}
\definecolor{magenta}{rgb}{1,0,1}
\definecolor{dartmouthgreen}{rgb}{0.05, 0.5, 0.06}
\definecolor{purple(x11)}{rgb}{0.63,0.36,0.94}
\definecolor{turquoise}{rgb}{0.25, 0.87, 0.81}
\newtheorem{theorem}{Theorem}[section]
\newtheorem{lemma}[theorem]{Lemma}
\newtheorem{proposition}[theorem]{Proposition}
\newtheorem{corollary}[theorem]{Corollary}

\theoremstyle{definition}
\newtheorem{remark}[theorem]{Remark}
\newtheorem{example}[theorem]{Example}

\newcommand{\Log}{\mathrm{Log}}

\newcommand{\Cal}[1]{\ensuremath{\mathcal{#1}}}

\newcommand{\lp}{\left(}
\newcommand{\rp}{\right)}

\usepackage[textheight=8.75in, textwidth=6.75in]{geometry}

\def\C{{\mathbb C}}

\def\N{{\mathbb N}}
\def\Z{{\mathbb Z}}
\def\Q{{\mathbb Q}}

\def\O_K{{\Cal{O}_{K}}}
\def\O_F{{\Cal{O}_{F}}}
\def\N_F{{\Cal{N}_{F/\Q}}}

\def\O_K{{\Cal{O}_{K}}}
\def\O_F{{\Cal{O}_{F}}}
\def\N_F{{\Cal{N}_{F/\Q}}}

\numberwithin{equation}{section}
\numberwithin{theorem}{section}

\title{Distribution of Hooks in Self-Conjugate Partitions}

\author{William Craig, Ken Ono, and Ajit Singh}
\address{Department of Mathematics, United States Naval Academy, 572C Holloway Road
Mail Stop 9E. Annapolis, MD 21402}
\email{wcraig@usna.edu}
\address{Dept. of Mathematics, University of Virginia, Charlottesville, VA 22904}
\email{ko5wk@virginia.edu}
\email{ajit18@iitg.ac.in} 
\begin{document} 
	
 \begin{abstract} We confirm the speculation that the distribution of $t$-hooks among unrestricted integer partitions essentially descends  to self-conjugate partitions. Namely,
 we prove that the number of hooks of length $t$ among the size $n$ self-conjugate partitions is asymptotically normally distributed with mean
 $\mu_t(n)$ and variance $\sigma_t(n)^2$ 
 $$\mu_t(n) \sim \frac{\sqrt{6n}}{\pi} + \frac{3}{\pi^2} - \frac{t}{2}+\frac{\delta_{t}}{4}\ \ \ {\text {\rm  and}}\ \ \  \sigma_t^2(n) \sim \frac{\lp \pi^2 - 6 \rp \sqrt{6n}}{\pi^3},
 $$
 where $\delta_t:=1$ if $t$ is odd and is 0 otherwise.
 \end{abstract}
\maketitle

\section{Introduction and Statement of results}

A {\it partition} of a non-negative integer $n$ is a non-increasing sequence of positive integers $\lp \lambda_1, \dots, \lambda_\ell \rp$ whose terms sum to $n$. We write $\lambda \vdash n$ to denote that $\lambda$ is a partition of $n$. Partitions play an  important role in many areas of mathematics, including combinatorics, geometry, mathematical physics, number theory and representation theory. Here we study the combinatorial statistics of partition hook numbers.

For a partition $\lambda$, integers $j,k \geq 1$ and any cell $\lp j,k \rp$ in the Young diagram of $\lambda$, the corresponding {\it hook number} $h(j,k)$ is the length of the {\it hook} $H(j,k)$ formed with the cell $\lp j,k \rp$ as its upper corner. In terms of the conjugate partition $\lambda' = \lp \lambda_1', \dots, \lambda_r' \rp$, we may write $h(j,k) = \lp \lambda_j - j \rp + \lp \lambda_k' - k \rp + 1$. Below, we see an example demonstrating the computation of hook numbers.
\begin{figure}[H] \centering $$\begin{ytableau} 7&6&4&3&1 \\ 5&4&2&1 \\ 2&1 \end{ytableau}$$
\caption{Hook numbers for the partition $\lambda = \lp 5, 4, 2 \rp$.} \vspace{-0.05in}
\end{figure}

Hook numbers play a significant role in the representation theory of the symmetric group, where the partitions of $n$ capture the irreducible representations of $S_n.$ Indeed, if
$\mathcal H(\lambda)$ is the multiset of hook lengths in $\lambda$ and $\rho_\lambda$ is the irreducible representation of $S_n$ associated with $\lambda$, then the Frame--Thrall--Robinson hook length formula gives the dimension
\begin{align*}
    \mathrm{dim}\lp \rho_\lambda \rp = \dfrac{n!}{\prod_{h \in \mathcal H(\lambda)} h}. 
\end{align*}
Furthermore, hook numbers are prominent in mathematical physics and number theory, For example, we highlight the work of Nekrasov and Okounkov \cite{NekOk} and Westbury \cite{Westbury}, who recognized the deep properties of hook numbers through their extraordinary $q$-series identity
\begin{align} \label{N-O Formula}
    \sum_{\lambda} q^{|\lambda|} \prod_{h \in \mathcal H(\lambda)} \lp 1 - \dfrac{z}{h^2} \rp = \prod_{n \geq 1} \lp 1 - q^n \rp^{z-1}.
\end{align}
Using this formula and its generalizations due to Han \cite{Han}, many connections have been drawn between hook numbers and modular forms, which have led to many interesting results, including theorems about cranks for Ramanujan's partition congruences \cite{GKS} and class numbers of imaginary quadratic fields \cite{OnoSze}.

Establishing combinatorial statistics for partitions is an important and growing field in partition theory (see for instance  \cite{AS,BCOM,BM,GOT,GORT}).  Here we consider the statistics of hook numbers.  Recently, Griffin, Tsai and the second author \cite{GOT} studied the counting function
\begin{align*}
	N_t(\lambda) := \#\{ h \in \mathcal H(\lambda) : h = t \}.
\end{align*}
For example, by enumerating the 11 partitions of 6, one finds that $N_{2,6}$ takes on the value $0$ with probability $\frac{1}{11}$, value $1$ with probability $\frac{4}{11}$, and value $2$ with probability $\frac{6}{11}$, where $N_{t,n}$ is the random variable which takes the value $N_t(\lambda)$ for $\lambda$ a random partition of $n$. Then \cite{GOT} proved the following theorem.

\begin{theorem}[{\cite[Theorem 1.1]{GOT}}] \label{T: GOT Theorem}
	For $t \geq 1$ an integer, the function $N_{t,n}$ has an asymptotically normal distribution as $n \to \infty$ with mean asymptotic to $\frac{\sqrt{6n}}{\pi} - \frac t2$ and variance asymptotic to $\frac{\lp \pi^2 - 6 \rp \sqrt{6n}}{2\pi^3}$. 
\end{theorem}

It is natural to ask whether the same phenomenon holds for the restricted partition families. In this paper, we show that this is essentially the case for self-conjugate partitions. To make this precise, for integers $t \geq 1$ we study the arithmetic statistics of $N_t(\lambda)$ considered as a random variable restricted to the class $\mathcal{SC}$ of self-conjugate partitions. Such a study requires the two-variable generating function which simultaneously tracks the size and hook counts of self-conjugate partitions; that is, we require an explicit formula for
\begin{align}
    F_t(T;q) := \sum_{\lambda \in {\mathcal SC}} T^{N_t(\lambda)} q^{|\lambda|} =: \sum_{n \geq 0} \mathrm{sc}_t(n;T) q^n.
\end{align}

By means of the Littlewood bijection for $t$-core partitions, Amdeberhan, Andrews and two of the authors \cite{AAOS} derived such a formula in order to address conjectures of the first author and collaborators \cite{BBCFW,CDH} on the arithmetic of hook counts in self-conjugate partitions.
In this way we obtain the following direct analog of Theorem \ref{T: GOT Theorem}. Throughout this paper, we define 
\begin{equation}
\delta_t:=\begin{cases} 1 \ \ \ \ &{\text {\rm if $t$ is odd,}}\\
0 \ \ \ \ &{\text {\rm otherwise.}}
\end{cases}
\end{equation}

\begin{theorem} \label{T: Main Theorem}
    Let $t \geq 1$ and consider the random variable $\widehat N_{t,n}$ giving the distribution of $N_t(\lambda)$ on the set ${\mathcal SC}(n)$ of self-conjugate partitions of $n$. Then as $n \to \infty$, $\widehat N_{t,n}$ is asymptotically normally distributed with mean $\mu_t(n) = \frac{\sqrt{6n}}{\pi}-\frac{t}{2}+\frac{3}{\pi^2}+\frac{\delta_{t}}{4} +O\left(\frac{1}{\sqrt{n}}\right)$ and variance  $\sigma_t(n)^2 = \frac{(\pi^2-6)\sqrt{6n}}{\pi^3}+\frac{3(\pi^2-12)}{\pi^4}-\frac{\delta_{t}}{8}  + O\left(\frac{1}{\sqrt{n}}\right)$.
\end{theorem}

\begin{remark}
It is interesting to note that, in comparison to Theorem \ref{T: GOT Theorem}, the main term of the mean is the same, but the main term of the variance is doubled in the self-conjugate case.
\end{remark}

\begin{example}
    In the case of $t=2$, Theorem \ref{T: Main Theorem} says that as $n \to \infty$ the number of 2-hooks in a random self-conjugate partition is asymptotically normal with mean $\mu_2(n) \sim \frac{\sqrt{6n}}{\pi} - 1 + \frac{3}{\pi^2}$ and variance $\sigma_2(n)^2 \sim \frac{(\pi^2-6) \sqrt{6n}}{\pi^3} + \frac{3(\pi^2-12)}{\pi^4}$. The convergence of this distribution in the case $t=2$ is demonstrated below.
    \begin{center}
    \begin{figure}[h!]
    	\vspace{-5pt}
    	\centering
    	\includegraphics[scale=0.20]{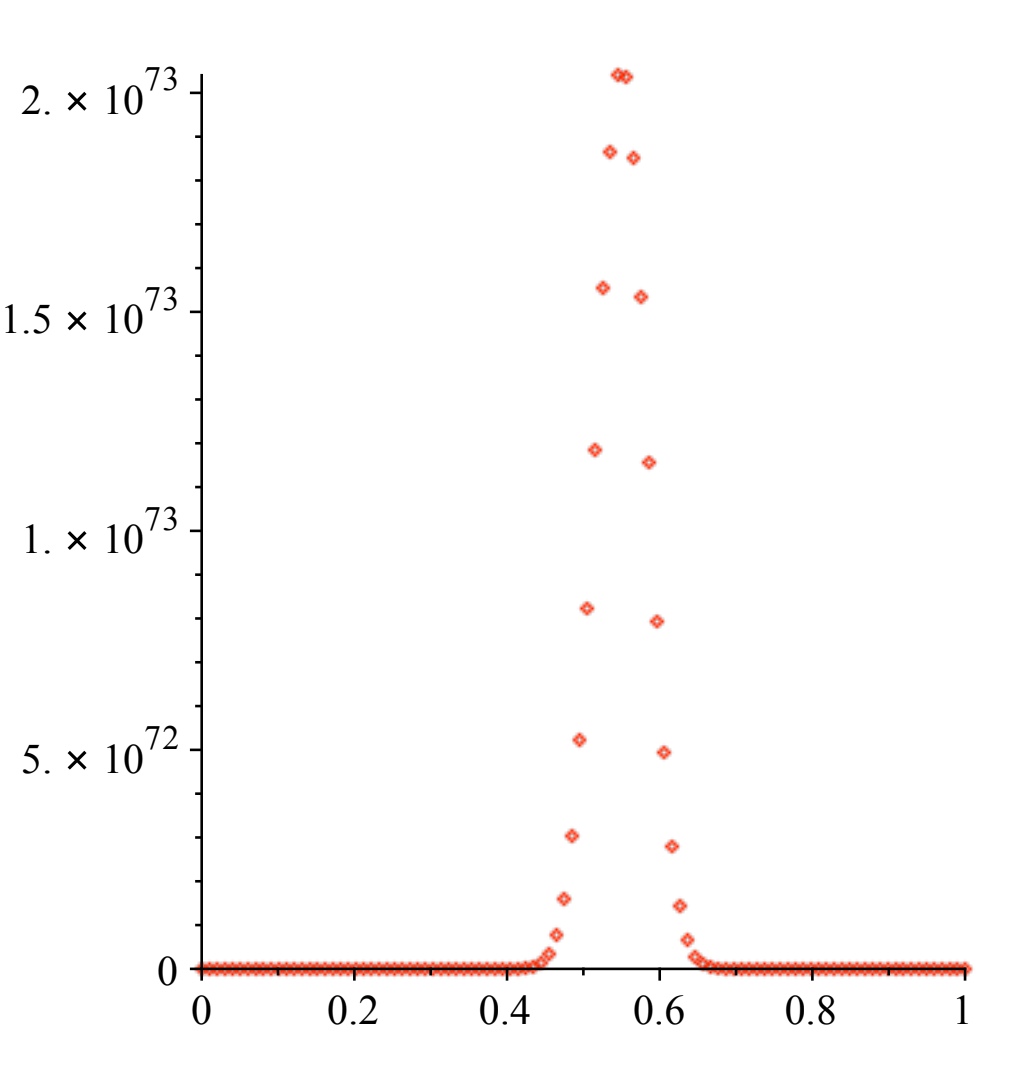}
    	\caption{Renormalized plot of the coefficients of $\mathrm{sc}_2(5000;T)$.}
    	\label{fig:example}
    	\vspace{-0.15in}
    \end{figure}
	\end{center}
	
	\noindent
		Table 1 shows the convergence of the measured means $\mu_2(n)$ to the asymptotic $\mu(n) := \frac{\sqrt{6n}}{\pi}$:
	\begin{table}[h]
		\begin{tabular}{|l|l|l|l|} \hline
			$n$ & $\mu_2(n)$ & $\mu(n)$ & $\mu_2(n)/\mu(n)$ \\ \hline
			100 & 9.17483 & 7.79697 & 1.17672 \\ \hline
			500 & 18.76417 & 17.43455 & 1.07626 \\ \hline
			1000 & 25.97841 & 24.65618 & 1.05363 \\ \hline
			5000 & 56.44511 & 55.13289 & 1.0238 \\ \hline
		\end{tabular}\smallskip
		
		\caption{Data showing the asymptotic $\mu_2(n) \sim \mu(n)$ as $n \to \infty$.}
	\end{table}
\end{example}

The paper is organized as follows. In Section \ref{S: Nuts and Bolts} we collect preliminary definitions and asymptotic lemmas we will need, which will be derived from the Euler--Maclaurin summation formula. In Section~\ref{S: Asymptotics} we apply various asymptotic lemmas and the saddle point method in order to obtain required asymptotics for the polynomials $\mathrm{sc}_t(n;T)$ for ranges of $T$ as $n \to \infty$. Finally, in Section \ref{S: Proof} we filter these asymptotics through the method of moments in order to prove Theorem \ref{T: Main Theorem}.

\section*{Acknowledgements}
	\noindent
	The authors thank the referees for helpful comments that improved the exposition and communication of the manuscript as well as for pointing out a mistake in the original manuscript. The first  author thanks the support of the European Research Council (ERC) under the European Union's Horizon 2020 research and innovation programme (grant agreement No. 101001179) and by the SFB/TRR 191 ``Symplectic Structure in Geometry, Algebra and Dynamics'', funded by the DFG (Projektnummer 281071066 TRR 191). The second author thanks the Thomas Jefferson Fund and grants from the NSF (DMS-2002265 and DMS-2055118). The third author is grateful for the support of a Fulbright Nehru Postdoctoral Fellowship. The views expressed in this article are those of the authors and do not reflect the official policy or position of the U.S. Naval Academy, Department of the Navy, the Department of Defense, or the U.S. Government.

\section{Nuts and Bolts} \label{S: Nuts and Bolts}

\subsection{Generating functions}

Here, we state the formula of \cite{AAOS} for the generating function $F_t(T;q)$. In order to state this formula, we need the notation
\begin{align*}
    \lp a;q \rp_n := \prod_{j=0}^{n-1} \lp 1 - a q^j \rp, \ \ \ n \in \mathbb{N}_0 \cup \{ \infty \}.
\end{align*}
With this notation, we state their result.

\begin{theorem}[{\cite[Theorem 1.1]{AAOS}}] \label{T: Generating function}
    Let $t \geq 1$ be an integer. Then the following are true:
    \begin{enumerate}
        \item If $t$ is even, then we have
        \begin{align*}
            F_t(T;q) = \lp -q;q^2 \rp_\infty \cdot \lp \lp 1 - T^2 \rp q^{2t}; q^{2t} \rp_\infty^{\frac t2}.
        \end{align*}
        \item If $t$ is odd, we have
        \begin{align*}
            F_t(T;q) = \lp -q;q^2 \rp_\infty \cdot H^*(T;q^t) \cdot \lp \lp 1 - T^2 \rp q^{2t}; q^{2t} \rp_\infty^{\frac{t-1}{2}},
        \end{align*}
        where $H^*(T;q)$ is defined by
        \begin{align} \label{E: H^* Definition}
            H^*(T;q) := \dfrac{1}{2T} \left[ C_T^+ \lp \sqrt{1-T^2} ; -q \rp_\infty + C_T^- \lp -\sqrt{1-T^2} ; -q \rp_\infty \right],
        \end{align}
        where for convenience we define the constants
        \begin{align*}
            C_T^\pm := 1 \pm \sqrt{\dfrac{1-T}{1+T}}.
        \end{align*}
    \end{enumerate}
\end{theorem}

We note  that $H^*$ can also be represented in terms of $q$-hypergeometric series \cite[Equation 3.1]{AAOS} as
\begin{align} \label{Hypergeometric series}
	H^*(T;q) = \lp 1 - \dfrac{1}{T} \rp \sum_{n \geq 0} \dfrac{(T^2-1)^n q^{2n^2 + n}}{\lp q^2; q^2 \rp_n \lp -q;q^2 \rp_{n+1}} + \dfrac{1}{T} \sum_{n \geq 0} \dfrac{(T^2-1)^n q^{2n^2-n}}{\lp q^2; q^2 \rp_n \lp -q;q^2 \rp_n}.
\end{align}
We primarily use the representation in terms of products because it is more convenient for much of our asymptotic analysis, although there are certain kinds of calculations which are easier in the $q$-hypergeometric form.

\subsection{The dilogarithm function}

We recall the {\it dilogarithm function} $\mathrm{Li}_2\lp z \rp$, which is given by
\begin{align*}
    \mathrm{Li}_2(z) := \sum_{k \geq 1} \dfrac{z^k}{k^2}
\end{align*}
for $|z| < 1$ and elsewhere by the standard analytic continuation having a branch cut on the line $(1,\infty)$. Dilogarithm functions appear frequently when computing asymptotic expansions of infinite products of the general shape $\lp a;q \rp_\infty$. For our calculations, we will need the elementary identity \cite[pg. 9]{ZagDilog}
\begin{align} \label{EQ Dilog ID}
    \mathrm{Li}_2\lp z \rp + \mathrm{Li}_2\lp -z \rp = \dfrac{\mathrm{Li}_2\lp z^2 \rp}{2},
\end{align}
as well as the derivative \cite[pg. 5]{ZagDilog}
\begin{align} \label{eq: Dilog derivative}
	\dfrac{d}{dz} \mathrm{Li}_2\lp z \rp = - \dfrac{\log(1-z)}{z}
\end{align}
and the so-called distribution property \cite[pg. 9]{ZagDilog}
\begin{align} \label{eq: distribution property}
	\mathrm{Li}_2\lp x \rp = n \sum_{z^n = x} \mathrm{Li}_2\lp z \rp.
\end{align}

\subsection{Euler--Maclaurin summation formulas}

As infinite products are often unwieldy to deal with directly, a common technique when analyzing $\lp a;q \rp_\infty$ asymptotically is to first take a logarithm and subsequently use techniques for infinite sums. In particular, we will make use of an asymptotic variation of the Euler--Maclaurin summation formula \cite{BJM, Zag}. To state this result, we will say that a function $f(z)$ satisfies the asymptotic $f(z) \sim \sum_{n \geq n_0} c_n z^n$ as $z \to 0$ if $f(z) - \sum_{n = n_0}^{N - 1 + n_0} c_n z^n = O_N\lp z^N \rp$ for each $N \geq 1$.

\begin{lemma}[{\cite[Theorem 1.2]{BJM}}] \label{L: Asymptotic EM}
	Suppose that $0\le \theta < \frac{\pi}{2}$ and let
	$$D_\theta := \{ re^{i\alpha} : r\ge0 \mbox{ and } |\alpha|\le \theta  \}.$$
    Suppose $f(z)$ is holomorphic in a domain containing $D_\theta$, and that $f$ and all its derivatives are of sufficient decay as $|z| \to \infty$ (i.e. decays at least as quickly as $|z|^{-1-\varepsilon}$ for some $\varepsilon>0$). Suppose also that $f$ has the asymptotic expansion $f(z) \sim \sum_{n \geq 0} c_n z^n$ near $z=0$. Then for any $0 < a \leq 1$, we have as $z \to 0$ that
    \begin{align*}
        \sum_{m \geq 0} f\lp (m+a)z \rp \sim \dfrac{1}{z} I_f - \sum_{n \geq 0} c_n \dfrac{B_{n+1}(a)}{n+1} z^n,
    \end{align*}
    where $I_f := \int_0^\infty f(x) dx$ and $B_n(x)$ are the Bernoulli polynomials.
\end{lemma}

We will apply Euler--Maclaurin summation many times throughout the paper, and in particular we will be interested in using not only a function $f$ as above, but several of its derivatives as well. In order to do this most efficiently, we derive the following corollary.

\begin{corollary} \label{C: EM with derivatives}
    Suppose that $0\le \theta < \frac{\pi}{2}$ and let $f:\C\rightarrow\C$ be holomorphic in a domain containing $D_\theta$, so that in particular $f$ is holomorphic at the origin, and assume that $f$ and all of its derivatives are of sufficient decay as $|z| \to \infty$. Then for $a\in\mathbb{R}$ and $j\in\mathbb{N}_0$, we have
    \begin{align*}
		\dfrac{d^j}{dw^j} \sum_{m\geq0} f((m+a)w) = \frac{(-1)^j j!}{w^{j+1}} I_f + O\lp 1 \rp.
	\end{align*}
\end{corollary}

\begin{proof}
    Since our decay assumptions justify the limit interchange below, we have
    \begin{align*}
        \dfrac{d^j}{dw^j} \sum_{m \geq 0} f\lp (m+a)w \rp = \sum_{m \geq 0} \lp m+a \rp^j f^{(j)}\lp (m+a)w \rp &= \dfrac{1}{w^j} \sum_{m \geq 0} g\lp (m+a)w \rp,
    \end{align*}
    where we let $g(w) := w^j f^{(j)}(w)$. It is then easy to show using integration by parts that $I_g = (-1)^j j! \cdot I_f$, and since $f^{(j)}(w)$ is holomorphic at $w=0$, we have $g(w) = O(w^j)$ which establishes the error term by the expansion in Lemma \ref{L: Asymptotic EM}.
\end{proof}

\subsection{Asymptotic lemmas}

In this subsection, we carry out a number of asymptotic estimates for $q$-series. We begin with some elementary applications of Euler--Maclaurin summation as presented above.

\begin{lemma} \label{L: (Xq;q) Lemmas}
     Let $0 \leq \theta < \frac{\pi}{2}$ and $D_\theta$ be as in Lemma \ref{L: Asymptotic EM}, and let $q = e^{-w}$, and let $X \in \C \backslash [1,\infty)$. Then we have for an integer $j \geq 0$ and $w \to 0$ in $D_\theta$ that
     \begin{align}
        \dfrac{d^j}{dq^j} \Log\lp -q;q^2 \rp_\infty &= \dfrac{\pi^2 j!}{24 w^{j+1}} + O\lp 1 \rp \label{eq0} \\
        \dfrac{d^j}{dq^j} \Log\lp Xq; q \rp_\infty &= - \dfrac{j! \cdot \mathrm{Li}_2\lp X \rp}{w^{j+1}} + O\lp 1 \rp, \label{eq1} \\
        \dfrac{d^j}{dq^j} \Log\lp X; -q \rp_\infty &= - \dfrac{j! \cdot \mathrm{Li}_2\lp X^2 \rp}{4 w^{j+1}} + O\lp 1 \rp, \label{eq2} \\
        \Log\lp Xq;q \rp_\infty &= - \dfrac{\mathrm{Li}_2\lp X \rp}{w} - \dfrac{\Log\lp 1 -X \rp}{2} + O(w), \label{eq3} \\
        \Log\lp X; -q \rp_\infty &= - \dfrac{\mathrm{Li}_2\lp X^2 \rp}{4w} + \dfrac{\Log\lp 1 - X \rp}{2} + O(w). \label{eq4}
     \end{align}
\end{lemma}

\begin{proof}
    We first observe that
    \begin{align*}
        \Log\lp -q;q^2 \rp_\infty = \sum_{m \geq 0} \Log\lp 1 + q^{2m+1} \rp = \sum_{m \geq 0} f\lp \lp m + \frac 12 \rp w \rp,
    \end{align*}
    where $f(w) := \Log\lp 1 + e^{-2w} \rp$. Using the fact that $\frac{d}{dq} = - e^w \frac{d}{dw}$, it is easy to show from Corollary \ref{C: EM with derivatives} that for any $j \geq 0$, we have
    \begin{align*}
        \dfrac{d^j}{dq^j} \Log\lp -q;q^2 \rp_\infty = \dfrac{j! \cdot I_f}{w^{j+1}} +{ O\lp 1 \rp}
    \end{align*}
    as $w \to 0$ in $D_\theta$, and so \eqref{eq0} follows from the fact that $I_f = \frac{\pi^2}{24}$.

    To prove \eqref{eq1}, we first observe
    \begin{align*}
        \Log\lp Xq;q \rp_\infty = \sum_{m \geq 0} f_X \lp \lp m+1 \rp w \rp,
     \end{align*}
     where
 $f_X\lp w \rp := \Log\lp 1 - X e^{-w} \rp.$
     We observe that
     \begin{align*}
        I_{f_X} := \int_0^\infty f_X(y) dy = - \mathrm{Li}_2\lp X \rp,
     \end{align*}
     the last identity being valid as long as $X \not \in [1,\infty)$. Applying \eqref{C: EM with derivatives} completes the proof of \eqref{eq1}. Note also that \eqref{eq3} follows from Lemma \ref{L: Asymptotic EM} after computing the series expansion of $f_X(w) = \Log(1-X) + O(w)$.

     To see \eqref{eq2}, we begin by noting that
     \begin{align*}
        \Log\lp X; -q \rp_\infty = \Log\lp 1 - X \rp + \sum_{m \geq 0} g_X^-\lp \lp m+1 \rp w \rp + \sum_{m \geq 0} g_X^+\lp \lp m + \dfrac 12 \rp w \rp,
     \end{align*}
     where
       $ g_X^\pm(w) := \Log\lp 1 \pm X e^{-2w} \rp.$
     Now, we have the integral evaluations, valid for $X \not \in [1,\infty)$,
     \begin{align*}
        I_{g_X^\pm} := \int_0^\infty \Log\lp 1 \pm X e^{-2y} \rp dy = - \dfrac{\mathrm{Li}_2\lp \mp X \rp}{2},
     \end{align*}
     and since the constant term of the Taylor expansion of $g_X^\pm(w)$ for all given $X$ is $\Log\lp 1 \pm X \rp$, and so by Lemma \ref{L: Asymptotic EM} we obtain
     \begin{align*}
        \Log\lp X; -q \rp_\infty = - \dfrac{\lp \mathrm{Li}_2\lp X \rp + \mathrm{Li}_2\lp - X \rp \rp}{2 w} + \dfrac{\Log\lp 1 - X \rp}{2} + O\lp w \rp.
     \end{align*}
    We can then derive \eqref{eq2} and \eqref{eq4} from \eqref{EQ Dilog ID} and Lemma \ref{L: Asymptotic EM} and Corollary \ref{C: EM with derivatives}.
\end{proof}

The previous lemma is a fairly standard application of Euler--Maclaurin summation; the fact that these generating functions are simple infinite products makes the method work very cleanly. For our setting, however, we are required to compute asymptotics for $H^*(T;q)$, which is a sum of two distinct infinite products. Since the logarithm of a sum is not especially well-behaved, the computation becomes more intricate. For this reason, we isolate these computations into the following lemma.

\begin{lemma} \label{L: H^* Lemma}
    Let $0 \leq \theta < \frac{\pi}{2}$ and $D_\theta$ be as in Lemma \ref{L: Asymptotic EM}, and let $q = e^{-w}$, and let $T>0$. Then we have as $w \to 0$ in $D_\theta$ that
    \begin{align}
        \Log\lp H^*(T;q^t) \rp &= - \dfrac{\mathrm{Li}_2\lp 1-T^2 \rp}{4tw} + \Log\lp \alpha_T^+ \rp + O(w), \label{H eq0} \\
        \dfrac{d}{dq} \Log\lp H^*(T;q^t) \rp &= \lp \dfrac{\mathrm{Li}_2\lp1-T^2\rp \alpha_T^-}{4 \sqrt{1-T^2} \alpha_T^+} \rp \dfrac{1}{t^2w^2} + O(w^{-1}), \label{H eq1} \\
        \dfrac{d^2}{dq^2} \Log\lp H^*(T;q^t) \rp &= O\lp w^{-4} \rp, \label{H eq2} \\
        \dfrac{d^3}{dq^3} \Log\lp H^*(T;q^t) \rp &= O \lp w^{-6} \rp \label{H eq3},
    \end{align}
    where we define the constants
    \begin{align*}
        \alpha_T^\pm &:= \frac{1}{2T}\left(C_T^+  \pm C_T^- \sqrt{\frac{1 + \sqrt{1-T^2}}{1 - \sqrt{1-T^2}}}\right).
    \end{align*}
\end{lemma}

\begin{proof}
    Recall that
    \begin{align*}
        H^*(T;q) :=  \dfrac{1}{2T} \left[ C_T^- \lp -\sqrt{1-T^2}; -q \rp_\infty + C_T^+ \lp \sqrt{1-T^2}; -q \rp_\infty \right].
    \end{align*}
     We set $X = \sqrt{1-T^2}$ for this proof; observe that for $T > 0$, $X \not \in [1,\infty)$, we obtain from \eqref{eq4} in Lemma~\ref{L: (Xq;q) Lemmas} that,     as $w \to 0$ in $D_\theta$, we have
    \begin{align} \label{eqn F_X asymptotic}
        \Log\lp \pm X; -q \rp_\infty = - \dfrac{\mathrm{Li}_2\lp X^2 \rp}{4 w} + \dfrac{\Log\lp 1 \mp X \rp}{2} + O\lp w \rp.
    \end{align}

    We now must move to understand the asymptotics of the derivatives of $\Log\lp H^*\lp T; q^t \rp \rp$, for which we will require derivatives of $H^*\lp T; q^t \rp$. As a convenient shorthand, we let $F_X := \lp X; -q \rp_\infty$ and $G_X := \sum\limits_{n \geq 1} \frac{(-1)^n n X q^{n-1}}{(-1)^n X q^n - 1}$ for the rest of the proof. Observe that Lemma \ref{L: (Xq;q) Lemmas} already gives enough for \eqref{H eq0}. It is then straightforward to show that
    \begin{align*}
        F_X^\prime = \sum_{n \geq 0} \dfrac{d}{dq}\left( 1 - X\lp -q \rp^n \right) \prod_{\substack{j \geq 0 \\ j \not = n}} \lp 1 - X \lp -q \rp^j \rp = F_X G_X,
    \end{align*}
    and so also we have
    \begin{align*}
        F_X^{\prime\prime} = F_X\lp G_X^2 + G_X^\prime \rp \ \ \ \text{and} \ \ \ F_X^{\prime\prime\prime} = F_X \lp G_X^3 + 3 G_X G_X^\prime + G_X^{\prime\prime} \rp.
    \end{align*}
    Now, if we set $H := H^*(T;q)$, we have $\frac{d}{dq} \Log\lp H \rp = \frac{H^\prime}{H}$ and
    \begin{align} \label{eq: diff id}
        \dfrac{d^2}{dq^2}\Log\lp H \rp = \dfrac{H H^{\prime\prime} - \lp H^\prime \rp^2}{H^2}, \ \ \ \dfrac{d^3}{dq^3}\Log\lp H \rp = \dfrac{H^2 H^{\prime\prime\prime} + 2\lp H^\prime \rp^3 - 3 H H^\prime H^{\prime\prime}}{H^3}.
    \end{align}
    In the previous notation, we have $H = C_T^+ F_X + C_T^- F_{-X}$ with $X = \sqrt{1 - T^2}$, and so we have
    \begin{align*}
        H^\prime &= C_T^+ F_X^\prime + C_T^- F_{-X}^\prime = C_T^+ F_X G_X + C_T^- F_{-X} G_{-X}, \\
        H^{\prime\prime} &= C_T^+ F_X^{\prime\prime} + C_T^- F_{-X}^{\prime\prime} = C_T^+ F_X\lp G_X^2 + G_X^\prime \rp + C_T^- F_{-X}\lp G_{-X}^2 + G_{-X}^\prime \rp, \\
        H^{\prime\prime\prime} &= C_T^+ F_X\lp G_X^3 + 3 G_X G_X^\prime + G_X^{\prime\prime} \rp + C_T^- F_{-X}\lp G_{-X}^3 + 3 G_{-X} G_{-X}^\prime + G_{-X}^{\prime\prime} \rp.
    \end{align*}
    We now compute asymptotics for the various derivatives of $H$. We observe that
    
    \begin{align*}
        G_{\pm X} = \dfrac{\pm X}{q} \left[ \dfrac{1}{w} \sum_{n \geq 0} f_{\pm X}^+\lp \lp n + \frac 12 \rp w \rp - \dfrac{1}{w} \sum_{n \geq 0} f_{\pm X}^-\lp \lp n+1 \rp w \rp \right], \ \ \text{where}~ f_X^\pm\lp w \rp := \dfrac{2w e^{-2w}}{1 \pm X e^{-2w}}.
    \end{align*}
    
    Thus, for $w \to 0$ in $D_\theta$, we have by Corollary \ref{C: EM with derivatives} and \eqref{EQ Dilog ID} that
    \begin{align*}
        G_{\pm X} &= \dfrac{1}{w^2} \lp \int_0^\infty f_{\pm X}^+(t) dt - \int_0^\infty f_{\pm X}^-(t) dt \rp + O(w^{-1}) = \pm \dfrac{1}{w^2}\lp \dfrac{\mathrm{Li}_2(X)}{2X} + \dfrac{\mathrm{Li}_2(-X)}{2X} \rp + O(w^{-1}) \\ &= \pm \dfrac{\mathrm{Li}_2(X^2)}{4Xw^2} + O(w^{-1}) = \pm \dfrac{\mathrm{Li}_2\lp 1 - T^2 \rp}{4Xw^2} + O(w^{-1})
    \end{align*}
    and additionally
    \begin{align*}
        G_{\pm X}^\prime = \mp \dfrac{\mathrm{Li}_2\lp 1 - T^2 \rp}{4Xw^3} + O\lp w^{-2} \rp, \ \ \ \ 
        G_{\pm X}^{\prime\prime} = \pm \dfrac{\mathrm{Li}_2\lp 1 - T^2 \rp}{2Xw^4} + O\lp w^{-3} \rp.
    \end{align*}
    Therefore, for $q = e^{-w}$ and $w \to 0$ in any region $D_\theta$, we have
    \begin{align*}
        G_{\pm X}^2 + G_{\pm X}^\prime &= O\lp w^{-4} \rp
    \end{align*}
    and likewise
    \begin{align*}
        G_{\pm X}^3 + 3 G_{\pm X} G_{\pm X}^\prime + G_{\pm X}^{\prime\prime} &= O(w^{-6}).
    \end{align*}
    We therefore obtain by use of \eqref{eqn F_X asymptotic} and the previous calculations the asymptotics
    \begin{align*}
        H &= F_X \lp \alpha_T^+ + O(w) \rp, \\
        H^\prime &= F_X \lp \dfrac{\mathrm{Li}_2\lp 1-T^2 \rp}{4Xw^2} \alpha_T^- + O(w^{-1}) \rp, \\
        H^{\prime\prime} &= F_X\cdot O(w^{-4}), \\
        H^{\prime\prime\prime} &= F_X \cdot O(w^{-6}).
    \end{align*}
    After the substitution $q \mapsto q^t$ in $H$ and after accounting for the chain rule, we obtain from these asymptotics and \eqref{eq: diff id} the desired asymptotic identities.
\end{proof}

\section{Asymptotics for $\mathrm{sc}(n;T)$} \label{S: Asymptotics}

In this section, we apply the saddle point method to compute asymptotics for the polynomial values $sc_t(n;T)$ as $n \to \infty$ for $T$ in certain ranges. This asymptotic behavior, we will see, determines the distributions in Theorem \ref{T: Main Theorem}.

\begin{proposition}\label{P: coeff asymp}
    Let $t$ be a positive integer, $\eta\in(0,1]$, and $\eta\leq T\leq \eta^{-1}$. If
    \begin{align*}
        b_t(T) := \begin{cases}
            \frac{1}{2}\sqrt{\frac{\pi^2}{6} - \mathrm{Li}_2(1-T^2)} & \text{if } 2|t, \\[+0.1in]
            \frac{1}{2}\sqrt{\frac{\pi^2}{6} - \lp \frac{t-1}{t} \rp \mathrm{Li}_2(1-T^2)} & \text{if } 2 \centernot | t,
        \end{cases}
    \end{align*}
    then
    \begin{align*}
        sc_t(n;T)=\sqrt{\frac{b_t(T)}{4\pi n^{3/2}}}e^{b_t(T)(2\sqrt{n}-\frac{1}{\sqrt{n}})}\lp 1+O_{\eta}\lp n^{-\frac{1}{7}}\rp \rp
    \end{align*}
    as $n \to \infty$.
\end{proposition}

\begin{proof}
    The proof will follow from the saddle point method. We have from Cauchy's theorem that
    \begin{align}\label{P:Eq1}
        sc_t(n;T)=\frac{1}{2\pi}\int_{-\pi}^{\pi}(z_0e^{ix})^{-n}F_t(T;z_0e^{ix})dx=\frac{1}{2\pi}\int_{-\pi}^{\pi}e^{f_t(T;z_0e^{ix})}dx,
    \end{align}
    where $f_t(T,z):=\Log(z^{-n}F_t(T;z))$ for $0<|z|<1$. In order to apply the saddle point method, we must determine $z_0=e^{-\alpha}$ for $\alpha>0$ such that $f_t'(T;z_0)=0$. Now, from Theorem \ref{T: Generating function}, $f_t(T;z)$ is given by
    {\small
    \begin{align*}
        \begin{cases}
            -n\Log(z) + \Log\lp -z;z^2 \rp_\infty + \dfrac{t}{2} \ \Log\lp \lp 1-T^2 \rp z^{2t}; z^{2t} \rp_\infty & \text{if } 2|t, \\[+0.1in]
            -n\Log(z) + \Log\lp -z;z^2 \rp_\infty + \dfrac{t-1}{2} \ \Log\lp \lp 1-T^2 \rp z^{2t}; z^{2t} \rp_\infty + \Log\lp H^*(T;z^t) \rp & \text{if } 2 \centernot | t.
        \end{cases}
    \end{align*}
	}
    We then have for $z = e^{-\alpha}$ from \eqref{eq0}, \eqref{eq1} and \eqref{H eq1} that
    \begin{align*}
        e^{-\alpha}f_t'(T;e^{-\alpha}) = \begin{cases}
            -n + \dfrac{\pi^2}{24 \alpha^2} - \dfrac{\mathrm{Li}_2\lp 1-T^2 \rp}{4 \alpha^2} + O(1) & \text{if } 2|t, \\[+0.1in]
            -n + \dfrac{\pi^2}{24 \alpha^2} + \dfrac{(t-1)}{t} \dfrac{\mathrm{Li}_2\lp 1-T^2 \rp}{4 \alpha^2} + O(\alpha^{-1}) & \text{if } 2 \centernot | t.
        \end{cases}
    \end{align*}
    Thus, the saddle point $z_0= e^{-\alpha}$ is given by\footnote{Although we can obtain a better error term for even $t$, we choose to treat the even and odd cases uniformly since the lower error terms are sufficient for our purposes.}
    \begin{align} \label{E: saddle point}
        \alpha=b_t(T) n^{-\frac{1}{2}} + O_{\eta}(n^{-\frac{3}{2}}).
    \end{align}

    We now estimate $f_t(T;z_0), f''_t(T;z_0)$, and $f'''_t(T;z_0)$. Putting $z_0=e^{-\alpha}$ in $f_t(T;z_0)$, we get using \eqref{eq0}, \eqref{eq2} and \eqref{H eq0} that in both cases, we have at the saddle point
    \begin{align} \label{E: f_t estimate}
        f_t(T;z_0) = 2 b_t(T) \sqrt{n} + O_\eta\lp n^{-\frac 12} \rp.
    \end{align}
    We now consider $f_t''$ and $f_t'''$. By comparing the contents of Lemma \ref{L: (Xq;q) Lemmas} with Lemma \ref{L: H^* Lemma}, particularly \eqref{H eq2} and \eqref{H eq3}, we see that the terms coming from $H^*(T;z^t)$ can only contribute to the error term, as they are smaller by an order of $O\lp n^{-\frac 12} \rp$, and therefore we may ignore these terms. Now, again using Lemma \ref{L: (Xq;q) Lemmas} we obtain
    \begin{align} \label{E: f_t'' estimate}
        f_t''(T;z_0) = e^{2b_t(T) n^{-1/2} + O_\eta(n^{-3/2})} \lp \dfrac{2 n^{3/2}}{b_t(T)} + O_\eta(n) \rp
    \end{align}
    and
    \begin{align} \label{E: f_t''' estimate}
        f_t'''(T;z_0) = O_\eta(n^2).
    \end{align}

    In order to complete the proof, we now let $sc_t(n;T)=I+II$, where
    \begin{align*}
        I:=\frac{1}{2\pi}\int_{|x|\leq n^{-5/7}}e^{f_t(T;z_o e^{ix})}dx~~~~~\text{ and }~~~~~II:=\frac{1}{2\pi}\int_{|x|> n^{-5/7}}e^{f_t(T;z_o e^{ix})}dx.
    \end{align*}
    To estimate $I$, we use the Taylor expansion of $f_t(T;z)$ centered at the saddle point $z_0=e^{-\alpha}$, given by
    \begin{align*}
        f_t(T;z)=f_t(T;z_0)+\frac{f''_t(T;z_0)(z-z_0)^2}{2}+O_\eta\lp f'''_t(T;z_0)\rp \cdot (z-z_0)^3.
    \end{align*}
    Since $|x|\leq n^{-5/7}$, the estimate \eqref{E: saddle point} implies 
    \begin{align*}
        z-z_0=z_0e^{ix}-z_0=e^{-\alpha}(ix+O(x^2))&=\left(1+O_\eta(n^{-1/2})\right)\left(ix+O_\eta(n^{-10/7})\right) \\ &=ix+O_\eta(n^{-17/14}).
    \end{align*}
    Hence, combining with \eqref{E: f_t''' estimate} we get
    \begin{align} \label{E: Taylor estimate}
        f_t(T;z)=f_t(T;z_0)-\frac{f''_t(T;z_0)}{2} x^2 + O_\eta(n^{-1/7}).
    \end{align}
    Combining \eqref{E: f_t estimate}, \eqref{E: f_t'' estimate}, \eqref{E: f_t''' estimate}, and \eqref{E: Taylor estimate} with classical integral evaluations, we obtain the asymptotic for $I$, namely
    \begin{align}\label{E: I estimate}
        I&= \frac{e^{f_t(T;z_0)}}{2\pi}\left[\int_{-\infty}^{\infty}e^{-\frac{f''_t(T;z_0)x^2}{2}}dx-2\int_{n^{-5/7}}^{\infty}e^{-\frac{f''_t(T;z_0)x^2}{2}}dx\right](1+O_\eta(n^{-1/7}))\nonumber\\
        &=\sqrt{\frac{b_t(T)}{4\pi n^{3/2}}}e^{b_t(T)(2\sqrt{n}-\frac{1}{\sqrt{n}})}(1+O_{\eta}(n^{-1/7})).    
    \end{align}

    To estimate the second integral $II$, we need a uniform estimate of $F_t(T;z)$ when $z$ is away from the saddle point. More precisely, we estimate the $\frac{F_t(T;z_0e^{ix})}{F_t(T;z_0)}$ using
        \begin{align*}
        e^{f_t(T;z_0 e^{-ix})} = e^{f_t(T;z_0)} \dfrac{F_t(T;z_0e^{ix})}{F_t(T;z_0)}.
    \end{align*}
    We set $z=z_0e^{ix}$, and we first consider the case of $t$ even. Since $T>0$, we have 
    \begin{align}\label{E: E upper bound}
        \left|\frac{F_t(T;z)}{F_t(T;z_0)}\right|^2&\leq \prod_{m\geq 1}max\left\{1,\left|\frac{1+(T^2-1)z^{2mt}}{1+(T^2-1)z_0^{2mt}}\right|^t\right\}\left|\frac{1+z^{2m-1}}{1+z_0^{2m-1}}\right|^2\nonumber\\
        &\leq \prod_{m\geq 1} E_{m,t}(z_0, T, x) \nonumber\\
        &\leq \prod_{\sqrt{n}\leq m\leq 2\sqrt{n}}E_{m,t}(z_0, T, x),
    \end{align}
    where for convenience we define
    {\small
    \begin{align*}
        E_{m,t}(z_0, T, x) := \max\left\{1,\left(1+\frac{2(1-T^2)z_o^{2mt}(1-\cos(2xmt))}{(1-z_0^{2mt})^2}\right)^{t/2}\right\}\left(1+\frac{2z_0^{2m-1}(\cos((2m-1)x)-1)}{(1+z_0^{2m-1})^2}\right).
    \end{align*}
    }

    Since $z_0^{\sqrt{n}}\to e^{-b_t(T)}$ and $t$ is fixed, we know that $z_0^{m}$ and $z_0^{mt}$ are bounded for $\sqrt{n}\leq m\leq 2\sqrt{n}$. Consequently, the same is true for $2z_0^{2mt}/(1-z_0^{2mt})^2$ and $2z_0^{2m-1}/(1+z_0^{2m-1})^2$. We consider two case (i.e. $T>1$ and $T\leq 1$) to estimate \eqref{E: E upper bound}. If $T>1$ and $\sqrt{n}\leq m\leq 2\sqrt{n}$ then by \eqref{E: saddle point} we have that $2z_0^{2m-1}/(1+z_0^{2m-1})^2\leq A_\eta$, for some $\eta>0$. This implies that 
\begin{align}\label{E: E upper bound II}
\left|\frac{F_t(T;z)}{F_t(T;z_0)}\right|^2\leq\prod_{\sqrt{n}\leq m\leq 2\sqrt{n}} \lp 1-A_\eta(1-\cos((2m-1)tx)) \rp^{t/2}.
\end{align}
A short computation also shows that \eqref{E: E upper bound II} still holds for $T\leq 1$ by choosing a suitable $A_\eta$.

We divide the range of $x$ into two parts $n^{-5/7}\leq |x|\leq \frac{\pi}{2\sqrt{n}}$ and $\frac{\pi}{2\sqrt{n}}\leq |x|\leq \pi$. For the first part, we use the inequality $1-\cos((2m-1)tx)\geq \frac{2}{\pi^2}((2m-1)tx)^2$ to estimate \eqref{E: E upper bound II}, to get
\begin{align}\label{E: E upper bound III}
\left|\frac{F_t(T;z)}{F_t(T;z_0)}\right|^2&\leq\prod_{\sqrt{n}\leq m\leq 2\sqrt{n}}\lp 1-\frac{2}{\pi^2}A_\eta((2m-1)tx)^2 \rp^{t/2}\leq\lp 1-A_\eta' n (tx)^2 \rp^{ t\sqrt{n}/2+O(1)}\nonumber\\
&\leq e^{-A_\eta' t^3x^2 n^{3/2}}\leq e^{-A_\eta' t^3 n^{1/14}}.
\end{align}
For the second part, we count the $m\in[\sqrt{n},2\sqrt{n}]$ for which there is an $r\in\Z$ with $-n^{-1/12}+2r\pi\leq xm\leq n^{-1/12}+2r\pi$. The total number of such $m$ is $\gg n^{1/2}+O(n^{5/12})$. Hence, we obtain 
\begin{align}\label{E: E upper bound IV}
\left|\frac{F_t(T;z)}{F_t(T;z_0)}\right|^2&\leq(1-A_\eta(1-\cos(n^{-1/12})))^{tn^{1/2}/2+O(n^{5/12})}\nonumber\\ &=\lp 1-A_\eta n^{-1/6}+O(n^{-1/3})\rp^{tn^{1/2}/2+O(n^{5/12})} \ll n^{-A_\eta n^{1/14}}.
\end{align}
By combining \eqref{E: E upper bound III} and \eqref{E: E upper bound IV}, we get the upper bound for the integral $II$, namely
\begin{align}\label{E: E final bound}
 II \ll \frac{1}{2\pi}\int_{|x|>n^{-5/7}}e^{f_t(T;z_0)}\frac{F_t(T;z_0e^{ix})}{F_t(T;z_0)}dx \ll_\eta e^{-\frac{b_t(T)}{\sqrt{n}}-A_\eta t^2 n^{1/14}}.
\end{align}
Since $sc_t(n;T)=I+II$, the proposition for $t$ even follows using \eqref{E: I estimate} and \eqref{E: E final bound}.

It now remains to perform the same calculation for the case $t$ odd. Observe that for the first two terms appearing in $F_t(T;z)$ for $t$ odd the calculations are exactly analogous, and therefore we have
\begin{align} \label{E: t odd E bound}
    \left| \dfrac{F_t(T;z)}{F_t(T;z_0)} \right|^2 \ll e^{-A_\eta n^{1/14}} \cdot \left| \dfrac{H^*(T;z^t)}{H^*(T;z_0^t)} \right|^2.
\end{align}
As before, we bound in the region $z = z_0 e^{ix}$ for $n^{-5/7} \leq |x| \leq \pi$.

It is now necessary to bound the quotient of $H^*$ functions above. We will derive such bounds using a slightly unusual application of Lemma \ref{L: Asymptotic EM}. Recall from the proof of Lemma \ref{L: (Xq;q) Lemmas} that for $X \in \C \backslash [1,\infty)$ and $z = e^{-w}$, we have as $w \to 0$ in $D_\theta$ that
\begin{align*}
	\Log\lp X; -z \rp_\infty = \Log\lp 1-X \rp + \sum_{m \geq 0} g_X^-\lp (m+1) w \rp + \sum_{m \geq 0} g_X^+\lp \lp m + \dfrac 12 \rp w \rp,
\end{align*}
where $g_X^\pm(w) := \Log\lp 1 \pm X e^{-2w} \rp$. The argument of Lemma \ref{L: Asymptotic EM} is then used to give good bounds for $H^*(T;z)$ for a small arc near $z=1$ after setting $X = \pm \sqrt{1-T^2}$. We now observe that we can also obtain good asymptotics for $H^*$ near any root of unity $\zeta_k^h := e^{\frac{2\pi i h}{k}}$ by shifting $e^{-w} \mapsto e^{-w + \frac{2\pi i h}{k}}$. Following the same elementary series arguments as in previous lemmas, we can write for $z = e^{-w+\frac{2\pi i h}{k}}$ that
\begin{align*}
	\Log\lp X;-z \rp_\infty = \Log\lp 1 - X \rp &+ \sum_{n \geq 0} \Log\lp 1 + X \zeta_k^{-h(2n+1)} e^{-(2n+1)w} \rp \\ &+ \sum_{n \geq 0} \Log\lp 1 - X \zeta_k^{-h(2n+2)} e^{-(2n+2)w} \rp.
\end{align*}
After separating the occurrences of each $k$th root of unity,
\begin{align*}
	\Log\lp X;-z \rp_\infty = \Log\lp 1 - X \rp &+ \sum_{j = 0}^{k-1} \sum_{m \geq 0} g^-_{X,\frac hk,j}\lp \lp m + \frac{j+1}{k} \rp kw \rp \\ &+ \sum_{j=0}^{k-1} \sum_{m \geq 0} g^+_{X, \frac hk, j}\lp\lp m + \frac{j+1}{k} \rp kw \rp,
\end{align*}
where we define
\begin{align*}
	g^\pm_{X,\frac hk,j}(w) := \Log\lp 1 \pm X \zeta_k^{-2(j+1)h} e^{-2w} \rp.
\end{align*}
Now, for each $g^\pm_{X,\frac hk,j}$ function, we have the integral
\begin{align*}
	\int_0^\infty g^\pm_{X,\frac hk,j}(x) dx = - \mathrm{Li}_2\lp X \zeta_k^{-2(j+1)h} \rp,
\end{align*}
and therefore we have as $z \to \zeta_k^h$ by Lemma \ref{L: Asymptotic EM} that
\begin{align*}
	\Log\lp X; -z \rp_\infty = - \dfrac{1}{2kw} \sum_{j=0}^{k-1} \mathrm{Li}_2\lp X \zeta_k^{-2(j+1)h} \rp - \dfrac{1}{2kw} \sum_{j=0}^{k-1} \mathrm{Li}_2\lp -X \zeta_k^{-2(j+1)h} \rp + O(1).
\end{align*}
For odd values of $k$, note that $\zeta_k^{-2(j+1)h}$ runs through each $k$th root of unity as $j$ runs from 0 to $k-1$. Therefore, by application of \eqref{eq: distribution property} we have for odd values of $k$ that
\begin{align*}
	\sum_{j=0}^{k-1} \mathrm{Li}_2\lp X \zeta_k^{-2(j+1)h} \rp = \dfrac{1}{k} \mathrm{Li}_2\lp X^k \rp, \ \ \ \sum_{j=0}^{k-1} \mathrm{Li}_2\lp -X \zeta_k^{2(j+1)h} \rp = \dfrac{1}{k} \mathrm{Li}_2\lp -X^k \rp.
\end{align*}
Then by application of \eqref{EQ Dilog ID}, we finally obtain
\begin{align*}
	\Log\lp X;-z \rp_\infty = - \dfrac{\mathrm{Li}_2\lp X^k \rp + \mathrm{Li}_2\lp -X^k \rp}{2k^2 w} + O(1) = - \dfrac{\mathrm{Li}_2\lp X^{2k} \rp}{4k^2 w} + O(1).
\end{align*}
Applying this calculation to the definition of $H^*(T;z)$ given in \eqref{E: H^* Definition} and the definition of $C_T^\pm$, we see that as $z \to \zeta_k^h$ for any root of unity with odd order, we have
\begin{align*}
	H^*(T;z) &= - \dfrac{C_T^+}{2T}\lp \dfrac{\mathrm{Li}_2\lp (1-T^2)^k \rp}{4k^2 w} \rp - \dfrac{C_T^-}{2T} \lp \dfrac{\mathrm{Li}_2\lp (1-T^2)^k \rp}{4k^2 w} \rp + O(1) \\ &= - \dfrac{\mathrm{Li}_2\lp (1-T^2)^k \rp}{4k^2 T w} + O(1).
\end{align*}
Because the odd order roots of unity are dense on the unit disk and because the regions $D_\theta$ are open intervals on the disk of radius $|z| = e^{-\alpha_0}$, we see that
\begin{align*}
	\max\limits_{\substack{k \geq 1 \\ k \text{ odd}}} \left[ - \dfrac{\mathrm{Li}_2\lp \lp 1-T^2 \rp^k \rp}{4k^2 T} \right] \dfrac{1}{w} + O(1)
\end{align*} 
gives an asymptotic upper bound on the size of $H^*(T;z)$ on the whole disk of radius $|z|$. For $\eta \leq T \leq \eta^{-1}$ the function $\lp 1-T^2 \rp^k$ is decreasing as a function of $k$ (since $k$ is odd), and by \eqref{eq: Dilog derivative} we see that $\mathrm{Li}_2\lp x \rp$ is an increasing function of $x$ on $(-\infty,1)$, and therefore $\mathrm{Li}_2\lp (1-T^2)^k \rp$ is a decreasing function of $T$ on $\eta \leq T \leq \eta^{-1}$. Because $\mathrm{Li}_2\lp x \rp \to - \infty$ as $x \to 1^-$, both numerator and denominator are optimized at $k = 1$, and so
\begin{align*}
	\max\limits_{\substack{k \geq 1 \\ k \text{ odd}}} \left[ - \dfrac{\mathrm{Li}_2\lp \lp 1-T^2 \rp^k \rp}{4k^2 T} \right] = - \dfrac{\mathrm{Li}_2\lp 1-T^2 \rp}{4T},
\end{align*}
Because we have shown that the maximal order of $H^*(T;z)$ is achieved in the region near $z=1$, it follows that
\begin{align*}
	\left| \dfrac{H^*(T;z^t)}{H^*(T;z_0^t)} \right|^2 \ll 1,
\end{align*}
and then by \eqref{E: t odd E bound} the desired result is complete for $t$ odd.

\end{proof}

\section{Proof of Theorem \ref{T: Main Theorem}} \label{S: Proof}

\subsection{Finding Mean and Variance}

Before beginning the proof of Theorem \ref{T: Main Theorem}, it is important to know the mean and variance of the distributions we consider. In particular, we show the following:

\begin{proposition} \label{P: Mean and Variance}
	The random variable $\widehat N_{t,n}$ on the space $\mathcal{SC}(n)$ has mean $\mu_t(n) \sim \frac{\sqrt{6n}}{\pi}-\frac{t}{2}+\frac{3}{\pi^2}+\frac{\delta_{t}}{4} $ and variance $\sigma_t(n)^2 \sim \frac{(\pi^2-6)\sqrt{6n}}{\pi^3}+\frac{3(\pi^2-12)}{\pi^4}-\frac{\delta_{t}}{8} $ as $n \to \infty$.
\end{proposition}

The proof of Theorem \ref{T: Main Theorem} itself will imply Proposition \ref{P: Mean and Variance}. For convenience, we give a sketch here of another method for calculating these values which is applicable and straightforward (i.e. requires no guesswork or {\it a priori} knowledge of the solution) even if the overall distribution is unknown.

\begin{proof}[Sketch of proof for Proposition \ref{P: Mean and Variance}]
	We use the standard notation $\mathbb{E}\lp \mathrm{sc}_t\lp n, \bullet \rp \rp$ and $\mathbb{V}\lp \mathrm{sc}_t\lp n, \bullet \rp \rp$ to denote the mean and variance of $\mathrm{sc}_t\lp n, m \rp$ as $m$ varies. Recall that
	\begin{align*}
		\sum_{n \geq 0} sc(n) \mathbb{E}\lp sc_t(n,\bullet) \rp q^n = \dfrac{\partial F_t(T;q)}{\partial T}\bigg|_{T=1}
	\end{align*}
	and likewise, using the identity $\mathbb{V}(X) = \mathbb{E}(X^2) - \mathbb{E}(X)^2$ for any random variable $X$, we have
	\begin{align*}
		\sum_{n \geq 0} sc(n) \mathbb{V}\lp sc_t(n,\bullet)\rp q^n = \left[ \lp T \dfrac{\partial}{\partial T} \rp^2 F_t(T;q) - \lp T \dfrac{\partial}{\partial T} F_t(T;q) \rp^2 \right] \bigg|_{T=1}
	\end{align*}
	Therefore, the asymptotic growth of $\mathbb{E}\lp sc_t(n,\bullet) \rp$ and $\mathbb{V}\lp sc_t(n,\bullet) \rp$ as $n \to \infty$ can be calculated from the growth of $sc(n)$ as well as the growth of the Fourier coefficients of $T$-derivatives of $F_t(T;q)$. Although there are a variety of cases to consider in our application, the basic idea is the same in all cases. One can compute directly formulas for $\frac{\partial^j F_t(T;q)}{\partial T^j} \big|_{T=1}$ as a $q$-series using elementary methods\footnote{To deal with the required formulas relating to $H^*(T;q)$, it is easiest to use the $q$-hypergeometric representations found in \eqref{Hypergeometric series}.}. These representations all take the form $$\dfrac{\partial^j F_t(T;q)}{\partial T^j} \bigg|_{T=1} = R_{j,t}(q) \lp -q;q^2 \rp_\infty$$ for some rational functions $R_{j,t}(q)$. One can then derive asymptotic expansions (for $q = e^{-w}$ and $w \to 0$ in relevant regions for $$\lp -q;q^2 \rp_\infty \sim \exp\lp \frac{\pi^2}{24w}-\frac{w}{24}\rp\lp1 + O\lp\exp\lp\frac{-\pi^2}{w}\rp\rp\rp$$ using Lemma \ref{L: Asymptotic EM} and for $R_{j,t}(q)$ using Laurent expansions. After verifying certain ``minor arc" conditions, which will be automatic because $\lp -q;q^2 \rp_\infty$ is modular (see for instance \cite{BJM}), the desired formulas follow from a standard application of Wright's circle method (see for example \cite[Proposition 1.8]{NR}).
\end{proof}

\subsection{Proof of Theorem \ref{T: Main Theorem}}

We recall the method of moments, as formulated in the following classical theorem of Curtiss.

\begin{theorem} \label{T: Method of Moments}
	Let $\{ X_n \}$ be a sequence of real random variables, and define the corresponding moment generating function
	\begin{align*}
		M_{X_n}(r) := \int_{-\infty}^\infty e^{rx} dF_n(x),
	\end{align*}
	where $F_n(x)$ is the cumulative distribution function associated with $X_n$. If the sequence $\{ M_{X_n}(r) \}$ converges pointwise on a neighborhood of $r=0$, then $\{ X_n \}$ converges in distribution.
\end{theorem}

The proof of Theorem \ref{T: Main Theorem} follows from Theorem \ref{T: Method of Moments} along with the theory of normal distributions. In particular, let $sc_t(n,m)$ be the number of self-conjugate partitions of $n$ having exactly $m$ hooks of length $t$ (i.e. the coefficient on $T^m$ in $\mathrm{sc}_t(n;T)$), and consider the $r$th power moments
\begin{align*}
	M_t\lp \widehat N_{t,n}; r \rp := \dfrac{1}{sc(n)} \sum_{m \geq 0} sc_t(n,m) e^{\frac{\lp m - \mu_t(n) \rp r}{\sigma_t(n)}}.
\end{align*}
By Theorem \ref{T: Method of Moments} and the theory of normal distributions, we need only prove that
\begin{align*}
	\lim\limits_{n \to \infty} M\lp \widehat N_{t,n}; r \rp = e^{\frac{r^2}{2}}.
\end{align*}

By evaluating $sc_t(n;T)$ at $T=1$ and $T=e^{\frac{r}{\sigma_t(n)}}$, we have that

\begin{align*}
	M\lp \widehat N_{t,n}; r \rp = \dfrac{sc_t\lp n;e^{\frac{r}{\sigma_t(n)}} \rp}{sc(n)} e^{-\frac{\mu_t(n)}{\sigma_t(n)}}.
\end{align*}

By Proposition \ref{P: coeff asymp}, we see that
\begin{align*}
	M\lp \widehat N_{t,n}; r \rp &= \sqrt{\dfrac{b_t\lp e^{\frac{r}{\sigma_t(n)}} \rp}{b_t(1)}} e^{\lp 2\sqrt{n} - \frac{1}{\sqrt{n}} \rp \lp b_t\lp e^{\frac{r}{\sigma_t(n)}} \rp - b_t(1) \rp - \frac{\mu_t(n)}{\sigma_t(n)}} \lp 1 + O_\eta\lp n^{-\frac 17} \rp \rp.
\end{align*}
Since $e^{\frac{r}{\sigma_t(n)}} > 0$ and approaches 1 as $n \to \infty$, we can remove the dependence on $\eta$ from the implied constant. By a direct calculation of the dilogarithm function, we see that $b_t(1) = \frac{\pi}{2\sqrt{6}}$ and 
\begin{align*}
	b_t\lp e^{\frac{r}{\sigma_t(n)}} \rp = \begin{cases}
		\dfrac{\pi}{2\sqrt{6}} + \dfrac{\sqrt{\frac 32} x}{\pi} + \dfrac{\sqrt{\frac 32} \lp \pi^2 - 6 \rp x^2}{2\pi^3} + O(x^3) & 2|t, \\
		\dfrac{\pi}{2\sqrt{6}} + \dfrac{\sqrt{\frac 32} (t-1) x}{\pi t} + \dfrac{\sqrt{\frac 32} (t-1)\lp \lp \pi^2 - 6 \rp t + 6 \rp x^2}{2\pi^3 t^2} + O(x^3) & 2 \centernot | t.
	\end{cases}
\end{align*}
Therefore, we conclude quickly from the construction of $\mu_t(n)$ and $\sigma_t(n)$ that
\begin{align*}
	M\lp \widehat N_{t,n}; r \rp = e^{\frac{r^2}{2} + o_r(1)}\lp 1 + O_r\lp n^{-\frac 17} \rp \rp.
\end{align*}
By taking $n \to \infty$, Theorem \ref{T: Main Theorem} follows.

\end{document}